%% file: TCLlocaldim10.tex
\documentclass[12pt]{article}
\input monstyle.tex

\usepackage{euscript,amscd,amsgen,amssymb,latexsym}

\input input.tex

\usepackage{subeqnarray}
\usepackage{pdfsync,graphicx,epsfig}
\usepackage{color,graphics}
\usepackage{epstopdf}
\usepackage{epic}
\usepackage{eepic}
\usepackage{stmaryrd}
\newcommand{\uc}{\underline{c}}
\newcommand{\oc}{\overline{c}}
\DeclareMathOperator{\diam}{diam}	
\DeclareMathOperator{\var}{var}
\DeclareMathOperator{\diag}{diag}

\newcommand{{\cstm}}{{\textit{$\phi$-cstm}}}
\newcommand{{\ecstm}}{{\textit{e$\phi$-cstm}}}
\newcommand{\B}{\CB}
\renewcommand{\C}{\CC}

\begin{document}

\title{Central Limit Theorem for dimension of Gibbs measures for skew expanding maps}
\author{Renaud Leplaideur\& Beno\^\i t Saussol \footnote{Laboratoire de Math\'ematiques,
UMR 6205,
 Universit\'e de
Bretagne Occidentale, 6 rue Victor Le Gorgeu BP 809 
F -  29285 BREST Cedex, renaud.leplaideur@univ-brest.fr \& benoit.saussol@univ-brest.fr}}
\maketitle

\abstract{We consider a class of non-conformal expanding maps on the $d$-dimensional torus. 
For an equilibrium measure of an H\"older potential, we prove an analogue of the Central Limit Theorem for the fluctuations of the logarithm of the measure of balls as the radius goes to zero.

An unexpected consequence is that when the measure is not absolutely continuous, then half of the balls of radius $\eps$ have a measure smaller than $\eps^\delta$ and half of them have a measure 
larger than $\eps^\delta$, where $\delta$ is the Hausdorff dimension of the measure.

We first show that the problem is equivalent to the study of the fluctuations of some Birkhoff sums. Then we use general results from probability theory as the weak invariance principle and random change of time to get our main theorem.

Our method also applies to conformal repellers and Axiom A surface diffeomorphisms and possibly to a class of one-dimensional non uniformly expanding maps. These generalizations are presented at the end of the paper.}

\medskip
\noindent
\date{{\bf Keywords}: Gibbs measure, expanding maps, dimension, Central limit Theorem.}

\tableofcontents

\section{Introduction}
\subsection{General background and motivations}

Let consider a $C^{1+\alpha}$ diffeomorphism $T$ acting on some compact Riemaniann manifold $X$. 
We can associate to each $T$-invariant probability $\mu$ several global quantities: the Kolmogorov entropy $h_{\mu}$, the Lyapunov exponents $\lambda_{\mu,1}<\lambda_{\mu,2}<\ldots<\lambda_{\mu,k}$ and the Hausdorff dimension $\delta_{\mu}$; the dimension $\delta_{\mu}$ being the infimum of all the Hausdorff dimensions of sets with positive $\mu$-measure.
Let us assume that the measure is hyperbolic, in the sense that no Lyapunov exponent is zero.

For the case of one dimensional maps, we recall that the Lyapunov exponent is defined by $\lambda_{\mu}:=\int\log|T'|d\mu$. Then, the relation between these three terms is 
$$h_{\mu}=\delta_{\mu}\lambda_{\mu}.$$
For the higher dimensional case,  the relation is (see \emph{e.g.} \cite{Ledrappier-Young1})
$$h_{\mu}=\sum_{i}\delta_{\mu,i}\lambda_{\mu,i}^+,$$
where $\lambda_{\mu,i}^+$ denotes $\max(0,\lambda_{\mu,i})$. The terms $\delta_{\mu,i}$ may be considered as \emph{intermediate unstable} dimensions and we have $\delta_{\mu}^{u}=\disp\sum_{i,\, \lambda_{\mu,i}>0}\delta_{\mu,i}$
(similarly  we have $\delta_{\mu}^s=\disp\sum_{i,\, \lambda_{i}<0}\delta_{\mu,i}$).
On the other hand, associated to the measure $\mu$, there is a notion of local (or pointwise) dimension. We set 
$$\delta_{\mu}(x):=\lim_{\eps\rightarrow 0}\frac{\log\mu(B(x,\eps))}{\log\eps}$$
whenever the limit exists.  Here $B(x,\eps)$ is the open ball of radius $\eps$ centered at $x$. 
It is known (see \cite{Ledrappier-Young2} and \cite{Barreira-Pesin-Schmeling}) that  for $\mu$-almost every point $x$, the pointwise dimension $\delta_{\mu}(x)$ exists, is equal to $\delta_{\mu}$ and
$\delta_{\mu}=\delta_{\mu}^{u}+\delta_{\mu}^s$.

In this article, we study the fluctuations in this convergence for some dynamical systems $(X,T,\mu)$. Namely, we prove a Central Limit Theorem 
\[
\frac{\log\mu(B(x,\eps))-\delta_{\mu}\log\eps}{\sqrt{-\log\eps}}\stackrel{\CD}{\Longrightarrow}\CN(0,\sigma^2).
\]

An unexpected consequence is that when $\sigma\neq0$, then half of the balls of radius $\eps$ have a measure smaller than $\eps^{\delta_\mu}$ and half of them have a measure larger than $\eps^{\delta_\mu}$ (See Corollary~\ref{cor:median}).

The proof of this central limit theorem requires us to work at the \emph{level of processes}. That is, at some point, we need a \emph{functional} central limit theorem. With a little additional effort we also get the functional version of the above central limit theorem, which is the statement of our main theorem that we will now present in detail. 

\subsection{Statement of the Main Theorem}

\subsubsection{The dynamics}
We consider the $d$ dimensional torus $\T^d=\left(\R_{\disp/\Z}\right)^d$.
We denote by $\pi_k$ the canonical projections $\pi_k(x_1,\ldots,x_d)=(x_1,\ldots,x_k)$.

\begin{definition}
\label{def-skew}
A map $T:\T^d\circlearrowleft$ is said to be a \emph{skew product} if it is of the form $T(x)=(f_1(x_1),f_2(x_1,x_2),\ldots,f_d(x_1,\ldots,x_d))$.
\end{definition}

We consider $T:\T^d\circlearrowleft$   a $\CC^{2}$ skew product. 
We assume that $T$ is (uniformly) expanding, in the sense that 
\[
\sup_x \| (d_x T)^{-1} \| <1 
\]

Consider a H\"older continuous function $\varphi:\T^d\rightarrow\R$ called the \emph{potential},
and define its pressure by
\[
P(\varphi) := \sup\left\{h_{\mu}+\int\varphi\,d\mu\right\},
\]
where the supremum is considered on the set of $T$-invariant probabilities.
In this setting the supremum is attained at a unique invariant measure $\mu_{\varphi}$, which is called the equilibrium state of $\varphi$. 

Note that considering such a potential, we can assume that the pressure is equal to zero. This can be realized easily replacing $\varphi$ by $\varphi$ minus the pressure. 

\subsubsection{Skorohod topology}

In this article we shall use the Skorohod topology. 
We refer to \cite{billingsley} chapter 3 for more global setting on this topology.
We denote by $\CD([0,1])$ the set of \emph{cadlag} (french acronym for right continuous with left hand limits) functions on $[0,1]$ endowed with the Skorohod topology:

Two functions $u$ and $v$ in $\CD([0,1])$ are $\rho$-close if there exists $\lambda:[0,1]\to[0,1]$ such that 
\begin{enumerate}
\item $\lambda(0)=0$ and $\lambda(1)=1$ and $\lambda$ is increasing;
\item $\forall\, t\in[0,1]$, $|\lambda(t)-t|\le \rho$,
\item $\forall\, t$, $|u(\lambda(t))-v(t)|\le\rho$.
\end{enumerate}
In other words,  $u$ and $v$ are $\rho$-close in  $\CD([0,1])$ if, 
up to a small change of times, the two functions are $\rho$-close. 
The main feature of the space $\CD([0,1])$ is that it allows discontinuous functions but is still separable.

\subsubsection{Main result and corollaries}

Our main theorem is

\medskip\noindent
{\bf  Main Theorem.}
{\it Let $T:\T^d\circlearrowleft$ be a skew product $\CC^2$ expanding map. 
Let $\varphi$ be a H\"older continuous function from $\T^d$ to $\R$ . 
Let $\mu_{\varphi}$ be the equilibrium state associated to $\varphi$. 
Let $\delta$ be its Hausdorff dimension.

We assume that the sequence 
\[
\lambda_{\mu,i} := \int \log\left|\frac{\partial f_i}{\partial x_i}\right|\circ\pi_i d\mu_\varphi, 
\quad i=1,\ldots,d
\]
is increasing. Then there exists a real number $\s\ge0$ such that the process 
$$
\frac{\log\mu_{\varphi}\left(B(x,\eps^{t})\right)-t\delta\log\eps}{\sqrt{-\log\eps}}
$$ 
converges in $\CD([0,1])$ and in distribution to the process $\sigma W(t)$, where $W$ is the standard Wiener process.

In addition, the variance $\sigma^2$ is zero if and only if $\mu_\varphi$ is the unique absolutely continuous invariant measure, or equivalently $\varphi$ is cohomologous to $-\log |\det DT|$.
}
\medskip

We emphasize that for the absolutely continuous invariant measure, the measure of balls is  completely governed by its density $h$ with respect to the Lebesgue measure: the density is continuous (in fact $C^1$), therefore we have the equivalence 
\[
\mu_\varphi(B(x,\eps)) \sim h(x) \eps^d
\]
for any $x\in\T^d$. Needless to say, there is no point in looking at fluctuations in this case.

\begin{corollary}[Central limit theorem]
With the same assumptions and notations, the family of random variables 
$$
\frac{\log\mu_{\varphi}\left(B(x,\eps)\right)-\delta\log\eps}{\sqrt{-\log\eps}}
$$ 
converges in distribution to the (possibly degenerate) gaussian distribution $\CN(0,\sigma^2)$.
\end{corollary}

An immediate consequence is the unexpected balance between ``heavy'' and ``light'' balls, already mentioned in the introduction:

\begin{corollary}[Median]\label{cor:median}
With the same assumptions and notations, if $\mu_\varphi$ is not absolutely continuous then 
$$
\mu_\varphi\left( \left\{ x\colon \mu_{\varphi}\left(B(x,\eps)\right)\le\eps^\delta\right\}\right)
\to \frac12.
$$
\end{corollary}

We emphasize that the CLT was the main goal of the paper, but the method, at the level of processes, gives as a byproduct several standard corollaries; we refer to~\cite{billingsley} for further precisions about functions of Brownian motion paths.

\begin{corollary}[Maximum and minimum]
With the same assumptions and notations, if $\mu_\varphi$ is not absolutely continuous then 
$$
\mu_\varphi\left( \forall t\in[0,1], \mu_{\varphi}\left(B(x,\eps^t)\right)\le \eps^{t\delta+b\sigma/\sqrt{-\log\eps}}\right) \to \CM(b),
$$
where 
$$
\CM(b)=P(\sup_{t\in[0,1]}W_t\le b)=1-\frac{4}{\pi}\sum_{k=1}^\infty \frac{(-1)^k}{2k+1}e^{-\pi^2(2k+1)^2/8b^2}.
$$
\end{corollary}

\begin{corollary}[Arc-sine law]
With the same assumptions and notations, if $\mu_\varphi$ is not absolutely continuous then, the family of random variables 
$$
\CT_\eps(x) := Leb\left(t\in[0,1]\colon\mu_{\varphi}\left(B(x,\eps^t)\right)\le\eps^{t\delta}\right)
$$ 
converges in distribution to the Arc-sine law (recall that $U$ follows the arc-sine law if $P(U\le u)=\frac{2}{\pi}\arcsin\sqrt{u}$).
\end{corollary}

\subsection{Steps of the proof and structure of the paper}

To clarify the exposition the proof will be made in the two-dimensional case. For convenience we will denote points in $\T^2$ by $(x,y)$ and assume that the map $T$ is of the form $T(x,y)=(f(x),g(x,y))$. We set $\pi(x,y)=x$.

The proof has two main steps. In a first part (Section~\ref{sec:reduc}), we use dynamical and ergodic arguments to reduce the problem to the study of the convergence of some process of the form (see Lemma \ref{lem:primprim})
\begin{equation}
\label{equ1-form-sympa}
\frac{S_{n_{\eps^t}}\phi_1 + S_{m_{\eps^t}}\phi_2}{\sqrt{-\log\eps}},
\end{equation}
where $n_{\eps}$ and $m_{\eps}$ are random ``times''. 

Then, in Section~\ref{section-proba-wip} we use arguments from Probability Theory to prove the convergence of this last process. These arguments are somehow general and independent of the functions $\phi_{1}$ and $\phi_{2}$. 

We mention that the use of the Skorohod topology is perhaps not necessary. It seems useful because the process we study is \emph{a priori} discontinuous. However, note that the limit process is \emph{a.e.} continuous. Therefore the convergence is uniform. Nevertheless, 
the space of cadlag functions endowed with the norm of uniform convergence is not separable, which may cause some troubles as pointed out by P. Billingsley in \cite{billingsley}. We thus preferred to work in $\CD([0,1])$. 

\medskip

Our method also applies to conformal repeller and Axiom A surface diffeomorphisms. These adaptations are presented in Section~\ref{sec:general}. Hypothesis of uniform expansion does not seem to be so crucial and we also discuss some possible extensions of our main result at the end of the paper.

\section{Reduction to a non-homogeneous sum of random variables}\label{sec:reduc}

\subsection{A fibered Markov partition}\label{subsec-parti-markov}

Given $(x_0,y_0)\in\T^2$ we denote $S_0=\{x_0\}\times\T\cup\T\times\{y_0\}$.
\begin{lemma}
\label{lem-partimarkov}
For any $(x_0,y_0)\in\T^2$, there exist a finite partition $\CR$ of $\T^2$ in Markov proper sets $R_{i}$ such that 
\begin{enumerate}
\item 
For each element $R_{i}$ of the partition, $T(R_{i})=\T^2$ and $T_{|\inte R_{i}}$ is one-to-one. 
\item 
$\pi(\inte R_{i})\cap \pi(\inte R_{j})=\emptyset$ or $\pi(\inte R_{i})= \pi(\inte R_{j})$.
\item
The boundary $\partial \CR$ is mapped to $T(\partial R_i)\subset S_0$ 
\item
$\CP=\pi(\CR)$ is a Markov partition for $f$.
\end{enumerate}
\end{lemma}

\begin{proof}
As the map $T$ is a local diffeomorphism, the map $f$ is also a local diffeomorphism of $\T$. Both are onto.  Thus they are coverings with finite covers. 

Denote by $P_i$'s the collection of the closure of the connected components of $\T\setminus f^{-1}(\{x_0\})$. Each $\inte P_{i}$ is mapped by $f$ one-to-one, $f(P_{i})=\T$ and $f(\partial P_i)=\{x_0\}$. 

Similarly, the closure of the connected components of $\T^2\setminus T^{-1}S_0$ defines a finite collection of sets $R_i$ which fulfill the hypotheses (see Figure~\ref{fig-markovparti}).
\begin{figure}[htbp]
\begin{center}
\includegraphics[scale=0.5]{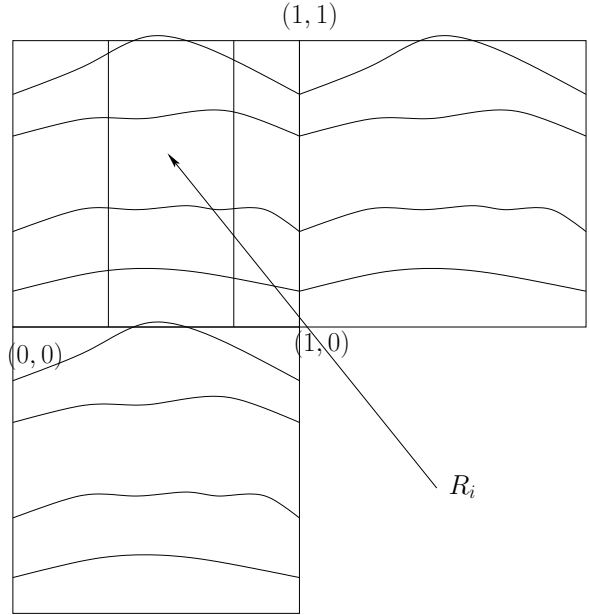}
\caption{Markov partition in nice proper sets}
\label{fig-markovparti}
\end{center}
\end{figure}
By construction, $T$ is one-to-one on $\inte R_{i}$ and $T(R_{i})=\T^2$. Now, for each $k$, $\pi(R_{k})$ is one of the $P_{i}$'s. These $P_{i}$'s have disjoint interior. 
\end{proof}

For $x$ in $\T$, $\CP_{n}(x)$ denote the element of the partition $\disp\bigvee_{k=0}^{n-1}f^{-k}(\CP)$ which contains $x$. Note that it is well defined up to the boundary of this ``partition''. Similarly we define $\CR_{n}(x,y)$.  By construction $\pi(\CR_{n}(x,y))=\CP_{n}(x)$. 

The border of the partition $\partial\CR_{n}$ is going to play an important role. For a fixed point $(x,y)$ and for an integer $n$, the border or $\CR_{n}(x,y)$ is denoted by $\partial\CR_{n}(x,y)$. It is the union of a vertical border $\partial^v\CR_{n}(x,y)$ and a horizontal border $\partial^{h}\CR_{n}(x,y)$. The vertical border is exactly the union of two vertical segments (its projection by $\pi$ is the union of two different points). The horizontal border is the union of two ``relatively'' horizontal curves. Their slope is studied in Lemma \ref{lem-slope-border}. 

We emphasize that the union over all integers of these borders is not an $T$-invariant set. In particular note that $T(\CR)$ has no boundary.

\subsection{Lyapunov exponents and geometry of the partition}
Given $f:\T\to\T$ and $g:\T^2\to\T$ two $\CC^1$ maps, we define for all integer $n$ 
\begin{equation}\label{eq:FnGn}
F_n=\prod_{j=0}^{n-1}f'\circ f^j\circ\pi,\quad
G_n=\prod_{j=0}^{n-1}\frac{\partial g}{\partial y}\circ T^j.
\end{equation}

\begin{lemma}
\label{lem-expo-vertifibred}
Let $T:\T^2\circlearrowleft$ be as in the Theorem. We set $T(x,y)=(f(x),g(x,y))$.
There is an invariant splitting $T\T^2=E^u\oplus E^{uu}$ defined $\mu$-a.e. 
The two associated Lyapunov exponents of $(T,\mu)$ are $\disp\lambda^u:=\int \log|f'(x)|\,d\mu_{\varphi}(x,y)$ and $\disp\lambda^{uu}:=\int\log\left|\deri{g}{y}(x,y)\right|\,d\mu_{\varphi}(x,y)$. 
\end{lemma}
\begin{proof}
By the ergodic theorem we have
\begin{equation}\label{eq:lyap}
\lim \frac{1}{n}\log F_n = \lambda^u < \lambda^{uu} = \lim\frac{1}{n}\log G_n.
\end{equation}
Therefore, the series 
\[
U=-\sum_{k=0}^\infty \frac{F_k}{G_{k+1}}\frac{\partial g}{\partial x}\circ T^k
\]
converges almost everywhere.
Define the splitting 
$$
E^u=\left(\begin{matrix} 1\\ U\end{matrix}\right),\quad E^{uu}=\left(\begin{matrix} 0\\1\end{matrix}\right).
$$
One directly checks that as announced the splitting is invariant:
$$
D_{(x,y)}T\left(\begin{matrix} 1\\ U(x,y)\end{matrix}\right)=f'(x)\left(\begin{matrix} 1\\ U\circ T(x,y)\end{matrix}\right),
\quad
D_{(x,y)}T\left(\begin{matrix} 0\\1\end{matrix}\right)=\frac{\partial g}{\partial y}(x,y)\left(\begin{matrix} 0\\1\end{matrix}\right).
$$
\end{proof}

We will need some estimates for the top and bottom borders $\partial^h\CR_{n}$ of the partition $\CR_n$. 
Note that if a point $(x,y)$ belongs to $\partial^h\CR_{n}$ then, it also belongs to $\partial^h\CR_{m}$ for every $m\ge n$. 
We denote by $\CT_{x,y,n}$ the slope of the tangent to $\partial^h\CR_{n}$ at $(x,y)$.

\begin{lemma}\label{lem-slope-border}
\label{lem-borders}
For every $n$ and for $\mu_{\varphi}$-almost every $(x,y)$ there exists a real number  $C_{\partial^h}(x,y)$ such that for every $(x',y')$ in $\partial^h\CR_{n}(x,y)$, , 
$$|\CT_{x',y',n}|\le C_{\partial^h}(x,y).$$
\end{lemma}

\begin{proof}
We assume that $(x,y)$ is  such that the invariant splitting is defined. 
For $(x',y')$ in $\partial^{h}\CR_{n}(x,y)$, we set 
$$U_{n}:= -\sum_{k=0}^{n-1}\frac{F_k(x')}{G_{k+1}(x',y')}\frac{\partial g}{\partial x}\circ T^k(x',y').$$
Set $(\alpha,\beta)=(D_{(x',y')} T^n)^{-1}(1, 0)$. Then $(\alpha,\beta)$ is tangent to $\partial^{h}\CR_{n}(x,y)$ at $(x',y')$. Moreover we get 
$$DT^n=\left(\begin{matrix} F_n & 0\\ -G_n U_n & G_n\end{matrix}\right).$$

Therefore the slope  of $\disp  \left(\begin{matrix}\alpha \\\beta\end{matrix}\right) $ in the canonical basis is 
$$
\beta/\alpha = U_n.
$$
The bounded distortion property shows that there exists a constant $C_T$ such that for all  $(x'',y'')\in\CR_n(x,y)$ we have 
$$\frac1{C_{T}}|U_n(x,y)|\le |U_n(x'',y'')|\le C_T |U_n(x,y)|.$$
We use this double inequality for $(x',y')$. 
Hence,  $|\CT_{x',y',n}|\le C_T|U_n(x,y)|$ holds.
Finally $U_n(x,y)$ converges to $U$ for a.e. $(x,y)$. It is thus bounded, and the lemma is proved.
\end{proof}

\subsection{Multi-temporal Markov approximation of balls}

\begin{definition}
\label{def-neps-meps}
Let $\eps$ be a positive real number. 

(i) We denote by $n_{\eps}(x,y)$  the largest integer $k$ such that 
$G_k(x,y)\eps\le1$

(ii) we denote by $m_{\eps}(x)$ the largest integer $k$ such that $F_k(x)\eps\le 1$.
\end{definition}

\begin{lemma}\label{lem:FlGn}
There exists some constant $c>0$ such that $c\le F_{m_\eps(x)}(x)\eps\le 1$ and $c\le G_{n_\eps(x,y)}(x,y)\eps\le 1$.
\end{lemma}
\begin{proof}
The inequalities follow directly from the definition and the fact that the functions $f'$ and $\frac{\partial g}{\partial y}$ are bounded from above and from below by a positive constant.
\end{proof}

\begin{lemma}\label{lem:mn}
For $\mu_\varphi$ a.e. point we have 
$\disp\lim_{\eps\to0} \frac{n_\eps}{-\log\eps}=\frac{1}{\lambda^{uu}}$ and 
$\disp\lim_{\eps\to0} \frac{m_\eps}{-\log\eps}=\frac{1}{\lambda^{u}}$.
In particular we have $n_{\eps}\ll m_{\eps}$  (as $\eps\to 0$) for $\mu_\varphi$ a.e. $(x,y)$.
\end{lemma}
\begin{proof}
This is an immediate consequence of Equation~\eqref{eq:lyap} in the proof of Lemma~\ref{lem-expo-vertifibred} and Lemma~\ref{lem:FlGn}.
\end{proof}

\begin{definition}
\label{def:ball}
We define the multi-temporal Markov approximation of a ball by 
$$
 C_{\eps}(x,y):=\CR_{n_{\eps}(x,y)}(x,y)\cap \pi^{-1}(\CP_{m_{\eps}(x)}(x)).
$$ 
\end{definition}
This set is in spirit an approximation of the ball $B((x,y),\eps)$. We shall discuss this fact now. 

\begin{lemma}\label{lem:Tneps} Let $(x,y)$ be fixed in $\T^2$. 
The map $T^{n_\eps(x,y)}$ is one-to-one from $\inte\CR_{n_\eps}\cap\{x\}\times\T$ to $f^{n_\eps}(x)\times(\T\setminus\{y_0\})$.
\end{lemma}
\begin{proof}
$T^{n_\eps}$ is one-to-one from the interior of the cylinder $\CR_{n_\eps}$ to $\T^2\setminus S_0$ and preserve vertical fibers.
\end{proof}

\begin{lemma}\label{lem:mvt}
There exists a constant $D>0$ such that 
$\diam\CP_{m_\eps}(x)\le D\eps$ and $\diam( \CR_{n_\eps}(x,y)\cap\{x\}\times\T)\le D\eps$.
\end{lemma}
\begin{proof}
The first assertion follows immediately from the mean value theorem, bounded distortion property, and the fact that $f^{m_\eps}$ is one-to-one on $\inte\CP_{m_\eps}$.

For the second one, a vertical segment based on $x$ and contained in $\CR_{n_\eps}(x,y)$ is expanded by $T^{n_\eps}$ by a factor $G_{n_\eps}(x,y')$ by the mean value theorem, for some $y'$ such that $(x,y')$ in $\CR_{n_\eps}(x,y)$. The conclusion follows by bounded distortion property,  Lemmas \ref{lem:FlGn} and \ref{lem:Tneps}.
\end{proof}

\medskip
In the rest of the paper we use vocabulary from the Probability Theory. Namely, we consider random constants and/or random processes. The random part depends on the point $(x,y)$ chosen in $\T^2$ with respect to the law $\mu_{\varphi}$. Constants are constant with respect to the parameter $\eps$. Processes are functions in $t\in[0,1]$.

\begin{lemma}
\label{lem:Cball}
There is a choice of $(x_0,y_0)\in\T^2$ such that the following holds:

There exists a constant $\uc<1$, positive almost everywhere, and a function $\oc_\eps>1$,
satisfying $\oc_\eps=_{0}O(|\log\eps|)$ almost everywhere, such that for any $\eps>0$,
\[
C_{\uc\eps}(x,y) \subset B((x,y),\eps) \subset C_{\oc_\eps\eps}(x,y).
\]
\end{lemma}

\begin{proof}
Let $(x',y')\in C_\eps(x,y)$. By the first assertion of Lemma~\ref{lem:mvt} we have $d(x,x')\le D\eps$. 

It follows immediately the second assertion of Lemma~\ref{lem:mvt} and Lemma~\ref{lem-borders} that $C_\eps(x,y)$ is included in a ``bow tie'' of vertical size less than $D\eps+2C_{\partial^h}(x,y)D\eps$ (see Figure~\ref{fig-bowtie}).
Hence for any $\eps>0$ we have
\[
C_\eps(x,y) \subset B((x,y),2D(1+C_{\partial^h}(x,y))\eps).
\]
Set $\disp \ul c:=\frac1{2D(1+C_{\partial^h}(x,y))}$. We have just proved that
$\disp C_{\ul c\eps}(x,y)\subset B((x,y),\eps)$ holds.

\begin{figure}[htbp]
\begin{center}
\includegraphics[scale=0.5]{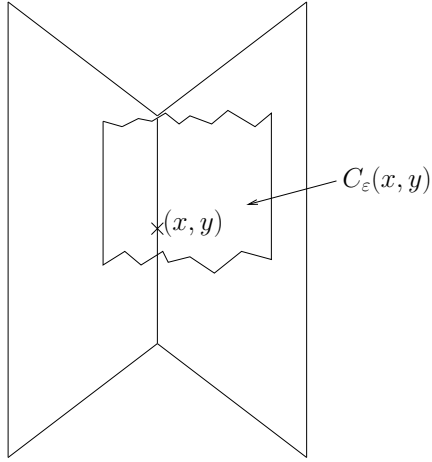}
\caption{The Markov approximation of the ball contained inside a Bow tie}
\label{fig-bowtie}
\end{center}
\end{figure}

\medskip
To get the other inclusion we need to control the distance between a point $(x,y)$ and the border of $C_{\eps}(x,y)$.

We claim that it is possible to choose $x_0$ and $y_0$ such that 
$$
\mu_\varphi(B(\partial \CR,r))\le ar,\quad \forall r>0
$$
where $a=8\|DT\|_\infty$.

Indeed, since $\mu_\varphi$ is a probability measure, there exist $x_0$ and $y_0$ such that 
$\mu_\varphi(B(x_0,r)\times\T )\le 4r$ and  for all $r$, $\mu_\varphi(\T\times B(y_0,r))\le 4r$ (see~\cite{Saussol-rapidmixing}, proof of Lemma~3 for details). 

We have 
$B(\partial \CR,r) = B(T^{-1}S_0,r) \subset T^{-1}B(S_0,\|DT\|_\infty r)$.
Hence by invariance of the measure we get $\mu_\varphi(B(\partial \CR,r))\le\mu_\varphi(B(S_0,\|DT\|_\infty r))\le ar$.

\medskip
Now, we show that for $\mu_{\varphi}$-almost every point  the orbit does not approach the border $\partial\CR$ too ``quickly'' . 

By Borel Cantelli Lemma and the invariance of $\mu_\varphi$ the claim implies that there exists $N=N(x,y)$, finite a.e., such that for any $n\ge N$ we have $d(T^n(x,y),\partial\CR)>1/n^2$.
In addition, the distance $d_N(x,y):=d((x,y),\partial \CR_N)$ is a.e. non zero since $\cup_{n=0}^{N}T^{-n}S_0$ has zero measure.

Note that $DT^n=\left(\begin{matrix} F_n & 0\\ -G_n U_n & G_n\end{matrix}\right)$.
Hence for $\mu_\varphi$-a.e. $(x,y)$ we have 
\[
\sup_{\CR_n(x,y)}|DT^n|\le \kappa(x,y)|G_n(x,y)|
\]
for some constant $\kappa>1$ finite a.e..

Let $\disp\rho_n=\frac{1}{n^2\kappa|G_n|}$. Let $n$ so large that $\rho_n<d_N(x,y)$. By induction we have that $B((x,y),\rho_n)\subset \CR_n(x,y)$. Indeed, suppose that for some $N\le k\le n-1$ we have 
$B((x,y),\rho_n)\subset \CR_{k}(x,y)$.
Since the image $T^kB((x,y),\rho_n)$ is contained in the ball $B((x,y),\kappa|G_k|\rho_n)$, which does not intersect the boundary $\partial \CR$, we get that $B((x,y),\rho_n)\subset \CR_{k+1}(x,y)$.

Taking $n=n_\eps$ (when $\eps$ is sufficiently small) we get that 
\[
B((x,y),\rho_{n_\eps})\subset \CR_{n_\eps}(x,y).
\]

\smallskip
A similar and easier argument applied to the one-dimensional map $f$ and the partition $\CP$ gives that for some sequence, say, $\disp\rho_m'=\frac{1}{m^2\kappa'|F_m|}$ we have
\[
B(x,\rho_{m_\eps}')\subset \CP_{m_\eps}(x).
\]

\medskip

Putting together these two inclusions, for any $\eps>0$ sufficiently small we get that

\begin{equation}
\label{equ1-inclus-cbareps}
B((x,y),\min(\rho_{n_\eps},\rho_{m_\eps}')) \subset C_\eps(x,y).
\end{equation}

To get the last inclusion, we rewrite \eqref{equ1-inclus-cbareps} with a variable $\alpha$ instead of $\eps$:
$$B((x,y),\min(\rho_{n_\alpha},\rho_{m_\alpha}')) \subset C_\alpha(x,y).
$$
Now, we want to inverse the expression in $\alpha$ and $\eps$: for a given $\eps$, 
there is $\alpha$ such $\min(\rho_{n_\alpha},\rho_{m_\alpha}')=\eps$. Hence 
$$B((x,y),\eps)\subset C_{\ol c_{\eps}.\eps}(x,y)$$
holds if we set $\ol c_{\eps}=\disp\frac\alpha\eps$. 

Note that we can always assume that the constant $\kappa$ and $\kappa'$ are bigger than 1. Hence, Lemma \ref{lem:FlGn} yields that $\alpha$ is (much) bigger than $\eps$. This shows that $n_{\eps}(x,y)$ and $m_{\eps}(x)$ are respectively bigger than $n_{\alpha}(x,y)$ and $m_{\alpha}(x)$. 

Assuming, for instance, that $\rho_{\alpha}=\eps$, we get 
$$\ol c_{\eps}=n^2_{\alpha}\kappa|G_{m_{\alpha}}|\alpha.$$
Again, we use Lemma \ref{lem:FlGn}, and then Lemma \ref{lem:mn} to get 
$$\ol c_{\eps}\le \wt\kappa(x,y)|\log\eps|,$$
for some constant $\wt\kappa$ \emph{a.e.} finite. 
\end{proof}

\begin{remark}
\label{rem-olc-bunded-above}
A direct consequence of Lemma \ref{lem:Cball} is that $\disp\frac{\log\ol c_{\eps}}{|\log^\frac14\eps|}$ is bounded from above when $\eps$ describes $[0,\frac12]$. 
\end{remark}

\subsection{The projected measure $\nu_{\varphi}$ is a Gibbs measure}\label{subsec-projection}
We define the projected measure $\nu_{\varphi}=\pi_*\mu_\varphi$ on $\T$ by 
$$\nu_{\varphi}(A):=\mu_{\varphi}(A\times \T).$$
As $T$ is a fibred map on $\T^2$ the measure $\nu_{\varphi}$ is $f$-invariant.  The goal of this subsection is to prove that $\nu_{\varphi}$ is a Gibbs measure. 

This comes from \cite{Chazottes-Ugalde-projmarko2}:

\begin{definition}[Amalgamation map]
\label{def-almagation}
Let $A, B$ be two finite alphabets, with $Card(A) > Card(B)$, 
and $\pi: A\rightarrow B$ be a surjective map (amalgamation) which extends to the map 
$ \pi:A^\N\rightarrow B^{\N}$ (we use the same letter for both) such that $(\pi \mathbf {a})_{n} = \pi(\mathbf{a}_{n} )$ for all $n\in\N$. The map $\pi$ is continuous and shift-commuting, i.e. it is a factor map from $A^\N$ onto $B^\N$. 
\end{definition}

We remind that the  the variation is \label{def-variation-n}$\var_n\phi=\sup_C\sup_{x,y\in C}|\phi(x)-\phi(y)|$ where the supremum is taken among all the cylinders $C$ of rank $n$.

\begin{theorem}[Chazottes-Ugalde]
\label{theo-projmarko}
Let $ \pi:A^\N\rightarrow B^{\N}$ be the amalgamation map just defined and 
$ \varphi:A^\N\rightarrow \R$ be a potential with exponentially decaying variation: $\var_{n}(\varphi)\in O(e^{-qn})$, for some $q>0$. Then the 
measure $\mu_{\varphi}\circ \pi^{-1}$ is a Gibbs measure with support $B^\N$, for a potential $\psi:B^\N\rightarrow \R$ with stretched exponential variation: $\var_{n}(\psi)\in O(e^{-c\sqrt{n}})$ for some $c>0$.
\end{theorem}

Using our vocabulary and our notation we get:

\begin{proposition}
\label{lem-nuphi-gibbs}
There exists a function $\psi$ which satisfies

(i) the variation of $\psi$ is stretched exponential.

(ii) the measure $\nu_{\varphi}$ is a Gibbs measure for $(\T,f)$ associated to the potential $\psi$.
\end{proposition}
\begin{remark}
\label{rem-pressionproj}
Without loss of generality we set the pressure of $\psi$ with respect to $(\T,f)$ to zero.
In particular we have $h_{\nu_\varphi}(f)=-\int\psi\circ\pi d\mu_\varphi$.

\end{remark}

\subsection{The measure of balls as Birkhoff sums}

For two random variables $a_\eps$ and $b_\eps$ we use the notation 
$a_\eps\approx b_\eps$ to mean that there exists a constant random variable $c<\infty$ a.e. such that $|a_\eps-b_\eps|\le c$ for any $\eps$.

Let us recall the definition of the main process
\[
N_\eps(t)=\frac{\log\mu_\varphi(B((x,y),\eps^t))-t\delta\log\eps}{\sqrt{-\log\eps}},
\quad t\in[0,1].
\]
By regularity of the measure, $N_\eps$ is cadlag\footnote{Presumably $N_{\eps}(t)$ is even continuous. However, the proof of that fact  would need more space than the margin allows us.}. We want to show the convergence of $N_\eps$ for the Skorohod topology on $[0,1]$.

\bigskip
Now, we define another process
$$
\quad N_\eps'(t)=\frac{\log\mu_{\varphi}(C_{\eps^t}(x,y))-t\delta\log\eps}{\sqrt{-\log\eps}},
\quad t\in[0,1].
$$
\begin{lemma}\label{lem-N-Nprime}
If the process $N_\eps'$ converges in distribution on $\CD([0,1])$ to a Wiener process of variance $\sigma^2$ then $N_\eps$ converges in distribution to the same process.
\end{lemma}
\begin{proof}
Observe that the process $N_\eps$ has the scale invariance
\[
N_{\eps}(t)=\sqrt{2}N_{\eps^2}(t/2),\quad \forall t\in[0,1].
\]
Since the Wiener process itself has the same scale invariance, and the mapping $w(\cdot)\mapsto \sqrt{2}w(\cdot/2)$ is continuous, it is sufficient to prove the convergence in distribution of the process $N_\eps$ on $\CD([0,1/2])$.

Let $\uc$ and $\oc_\eps$ given by Lemma~\ref{lem:Cball}.
For any $\eps<1/e^4$, on the set $\Omega_\eps^0:=\{\log\uc\ge -\log^{1/4}\frac1\eps\}$ and for any $t\le 1/2$ we have
\[
\begin{split}
N_{\eps}(t)
&\ge
\frac{\log\mu_\varphi(C_{\uc \eps^t}(x,y))-t\delta\log \eps}{\sqrt{-\log \eps}} \\
&\ge
\frac{\log\mu_\varphi(C_{\exp(-\log^{1/4}\frac1\eps) \eps^t}(x,y))-t\delta\log \eps}{\sqrt{-\log \eps}} \\
&\ge
N_{\eps}'(t+\log^{-3/4}\frac1\eps)-\delta \log^{-1/4}\frac1\eps =:U_\eps(t)
\end{split}
\]
since $\exp(-\log^{1/4}\frac1\eps)=\eps^{\log^{-3/4}\frac1\eps}$.

On the other hand, on the set  $\Omega_\eps^1=\{\log\oc_\eta\le\log^{1/8}\frac1\eps \log^{1/4}\frac1\eta,\forall \eta\in(0,\frac12)\}$, and for any $t\in[\log^{-5/8}\frac1\eps,1/2]$ we have
\[
\begin{split}
N_\eps(t)
&\le\frac{\log\mu_\varphi(C_{\oc_{\eps^t}\eps^t}(x,y))-\delta\log \eps}{\sqrt{-\log \eps}} \\
&\le
N_{\eps}'(t-\log^{-5/8}\frac1\eps)+\delta \log^{-1/8}\frac1\eps
\end{split}
\]
since\footnote{For $\eps<e^{-4}$ and for $t>\log^{-5/8}\frac1\eps$, $\eps^{t}\le e^{-4^{3/8}}=0.186..<\frac12$.} $\oc_{\eps^t}\le\exp(\log^{1/8}\frac1\eps \log^{1/4}\frac{1}{\eps^t}) \le \eps^{-\log^{-5/8}\frac1\eps}$.
Note in addition that for $t\in[0,\log^{-5/8}\frac1\eps)$, since $\mu$ is a probability measure, it trivially holds the upper bound 
\[
N_\eps(t)\le \frac{0-t\delta\log\eps}{\sqrt{-\log\eps}}\le\delta\log^{-1/8}\frac1\eps.
\]
Define 
\[
V_\eps(t) := \delta \log^{-1/8}\frac1\eps + \begin{cases} N_{\eps}'( t-\log^{-5/8}\frac1\eps) & \text{if } t\ge \log^{-5/8}\frac1\eps\\
0&\text{otherwise}
\end{cases}.
\]
For any $\eps<1/e^2$, on $\Omega_\eps^0\cap\Omega_\eps^1$ we have the bound on $[0,1/2]$:
\[
U_\eps\le N_\eps\le V_\eps.
\]

The measure of $\Omega_\eps^0\cap\Omega_\eps^1$ goes to $1$ (see Remark \ref{rem-olc-bunded-above} page \pageref{rem-olc-bunded-above}), and both $U_\eps$ and $V_\eps$ converge in distribution to the same process. We can now conclude the proof\footnote{The conclusion could follow from the sandwich theorem. However, a version for processes is not widely known, therefore we prove it directly in our case.}:

Denote, for any $q>0$, the oscillation of a function $w\in\CD([0,1])$ by $v(w,q)=\sup_{|t-s|<q} |w(t)-w(s)|$. We have 
\[
Z_\eps:= V_\eps-U_\eps \le v(N_\eps',2\log^{-5/8}\frac1\eps)+2\delta \log^{-1/8}\frac1\eps.
\]
Since $N_\eps'$ converges in distribution to a Wiener process $W$, which is continuous, we claim that the oscillation $v(N_\eps',2\log^{-5/8}\frac1\eps)$ converges to zero in probability: 

let $r>0$. Since $W$ is almost surely uniformly continuous, there exists $q>0$ such that $P(v(W,3q)>r/3)<r$. Let $A(q,r)=\{w\in \CD\colon v(w,q)>r\}$. The closure of $A(q,r)$ in the Skorohod topology is trivially contained in $A(3q,r/3)$. 
Moreover, the weak convergence of the measures $P_{N_\eps'}$ to $P_W$ implies 
$$\limsup_{\eps\to0} P_{N_\eps'}(A(q,r))\le P_{W}(A(3q,r/3))\le r.$$
Therefore, there exists $\eps_0$ such that for any $\eps<\eps_0$ we have
$P(v(N_\eps',q)>r)\le r+r$. Let $\eps_1<\eps_0$ such that $2\log^{-5/8}\frac1{\eps_1}<q$.
For any $\eps<\eps_1$ we have 
\[
P(v(N_\eps',2\log^{-5/8}\frac1\eps)>r)\le 2r.
\]
This proves the convergence in probability. 

By Slutsky theorem, $N_\eps$ also converges in distribution to the Wiener process.

\end{proof}
Therefore it suffices to show the convergence in distribution of the process $(N_\eps'(t))_{t\in[0,1]}$.
The key lemma below relates the measure of the multi-temporal Markov approximation of the ball with a non-homogeneous Birkhoff sum. This is where we use the skew product structure and the Gibbs property of the measure and its projection.
\begin{lemma}\label{lem:Cepssum}
For $\mu_\varphi$ a.e. $(x,y)$ we have 
\[
\log \mu_\varphi( C_\eps(x,y) ) \approx S_{n_\eps(x,y)}(\varphi-\psi\circ\pi)(x,y)+S_{m_\eps(x,y)}(\psi\circ\pi)(x,y)
\]

\end{lemma}
\begin{proof}
Remind that
$\disp C_{\eps}(x,y):=\CR_{n_{\eps}(x,y)}(x,y)\cap \pi^{-1}(\CP_{m_{\eps}(x)}(x))$.
Given $\eps_0>0$, set $\Omega(\eps_0):=\{(x,y)\in\T^2\colon \forall \eps\le \eps_0, m_\eps(x)\ge n_\eps(x,y)\}$. 
Let $(x,y)\in \Omega(\eps_0)$. In the following we omit the dependence with respect to $(x,y)$ in $n_{\eps}(x,y)$ and $m_{\eps}(x)$. 
Since $\mu_\varphi$ is $\exp(-\varphi)$ conformal and $T^{n_\eps}$ is 1-1 on $C_\eps$ we have 
$$
\mu_\varphi(T^{n_\eps} C_\eps) = \int_{C_\eps} e^{-S_{n_\eps}\varphi} d\mu_\varphi.
$$
Since $C_\eps$ is contained in the cylinder $\CR_{n_\eps}$, the bounded distortion property gives
$$
\log \mu_\varphi(C_\eps) \approx S_{n_\eps}\varphi(x,y) + \log\mu_\varphi(T^{n_\eps}C_\eps) 
$$
on $C_\eps$. Moreover, Lemma~\ref{lem:Tneps} gives that $T^{n_\eps}C_\eps=T^{n_\eps}(\CR_{n_\eps}\cap\pi^{-1}(\CP_{m_\eps}))=\pi^{-1}(f^{n_\eps}\CP_{m_\eps})$ and by the Markov property of $(f,\CP)$ we get $f^{n_\eps}\CP_{m_\eps}(x)=\CP_{m_\eps-n_\eps}(f^{n_\eps}(x))$.
Therefore
\[
\log\mu_\varphi(T^{n_\eps}C_\eps)  
= \log \nu_\varphi(\CP_{m_\eps-n_\eps}(f^{n_\eps}(x)))
\approx S_{m_\eps-n_\eps}\psi\circ f^{n_\eps}(x)
\]
by the Gibbs property of $\nu_\varphi$. We end up with
\[
\log \mu_\varphi(C_\eps) \approx S_{n_\eps}\varphi + S_{m_\eps-n_\eps}\psi\circ\pi\circ T^{n_\eps} 
= S_{n_\eps}(\varphi-\psi\circ\pi) + S_{m_\eps}\psi\circ\pi .
\]
This holds on $\Omega(\eps_{0})$. 
The conclusion follows since $\mu_\varphi(\Omega(\eps_0))\to1$ as $\eps_0\to0$ by Lemma~\ref{lem:mn}.

\end{proof}
Denote the intermediate entropies by $h^{uu}=h_{\mu_\varphi}(T)-h_{\nu_\varphi}(f)$ and $h^u=h_{\nu_\varphi}(f)$.
Since the pressures of $(T,\varphi)$ and $(f,\psi)$ are zero we get (see Remark~\ref{rem-pressionproj} page \pageref{rem-pressionproj} ) that
\begin{equation}\label{eq:lesentropies}
h^u=-\int\psi\circ\pi d\mu_\varphi,
\quad
h^{uu}=-\int(\varphi-\psi\circ\pi) d\mu_\varphi.
\end{equation}
\begin{lemma}\label{lem:pdimexists}
With the previous notation, we get the next formula for the pointwise dimension: 
\[
\frac{h^{uu}}{\lambda^{uu}}+ \frac{h^u}{\lambda^{u}} = \delta.
\]
\end{lemma}
\begin{proof}
It follows from Lemmas~\ref{lem:mn} and~\ref{lem:Cepssum} that $\mu_\varphi$-a.e.
\[
\begin{split}
\lim_{\eps\to0}\frac{\log\mu_\varphi(C_\eps)}{\log\eps} 
&=
\lim_{\eps\to0} \frac{n_\eps}{\log\eps} \frac{1}{n_\eps}S_{n_\eps}(\varphi-\psi\circ\pi)+
\frac{m_\eps}{\log\eps} \frac{1}{m_\eps}S_{m_\eps}(\psi\circ\pi) \\
&=
-\frac{1}{\lambda^{uu}}\int(\varphi-\psi\circ\pi) d\mu_\varphi -\frac{1}{\lambda^{u}}\int\psi\circ\pi d\mu_\varphi.
\end{split}
\]
Here, we recover that the pointwise dimension of the measure $\mu_\varphi$ exists $\mu_{\varphi}$-a.e. and is constant. This together with Equation~\eqref{eq:lesentropies} prove the first equality.
Since it is constant, it is necessarily the Hausdorff dimension $\delta$ of the measure $\mu_\varphi$.
\end{proof}

Set $\delta^{uu} = \frac{h^{uu}}{\lambda^{uu}}$, $\delta^u=\frac{h^{u}}{\lambda^{u}}$ and define the functions 
\begin{equation}\label{eq:lesphii}
\phi_1 = \varphi-\psi\circ\pi+\delta^{uu}\log\frac{\partial g}{\partial y},
\quad 
 \phi_2=\psi\circ\pi+\delta^u\log f'\circ\pi. 
\end{equation}
By Equation~\eqref{eq:lesentropies} and Lemma~\ref{lem-expo-vertifibred} we have  
\[
\int\phi_1 d\mu_\varphi=\int\phi_2d\mu_\varphi=0.
\]
\begin{proposition}\label{pro:varzero-acim}
If the functions $\phi_1$ and $\phi_2$ are both cohomologous to zero then $\varphi$ is cohomologous to $-\log |\det DT|$, and reciprocally.
\end{proposition}
\begin{proof}
Suppose that $\phi_1$ and $\phi_2$ are cohomologous to zero.

Since $\phi_2$ is $T$-cohomologous to zero, $\psi-\delta^u \log|f'|$ is $f$-cohomologous to zero, hence $\psi$ is $f$-cohomologous to $-\delta^u\log|f'|$. Therefore the $f$-pressure of $-\delta^u\log|f'|$ is zero. Since $f$ is uniformly expanding this implies that $\delta^u=1$.

We have that $\phi_1$ is cohomologous to $-\log|f'|-\delta^{uu}\log \left|\frac{\partial g}{\partial y}\right|$. Since $\det DT=f'\circ \pi \cdot \frac{\partial g}{\partial y}$ we get that
$\varphi$ is cohomologous to $-\log |\det DT| +(1-\delta^{uu})\log\left|\frac{\partial g}{\partial y}\right|$. But the convexity of the pressure gives
\[
0 = P_T(\varphi)\ge P_T(-\log|\det DT|)+(1-\delta^{uu})\int \log\left|\frac{\partial g}{\partial y}\right| d\mu_\varphi = (1-\delta^{uu})\lambda^{uu}.
\]
Therefore $\delta^{uu}\ge1$. On the other hand, $\delta^{u}+\delta^{uu}=\delta\le 2$, which implies that $\delta^{uu}=1$ also, proving the result.

The reciprocal is immediate.
\end{proof}

Define the process
\[
N_\eps''(t):=\frac{S_{n_{\eps^t}}\phi_1 + S_{m_{\eps^t}}\phi_2}{\sqrt{-\log\eps}},
\quad t\in[0,1].
\]
We are now able to relate the convergence of the two processes.

\begin{lemma}\label{lem:primprim}
There exists a constant $C_0<+\infty$ a.s. such that 
\[
\sup_{t\in[0,1]}\left| N_\eps'(t) - N_\eps''(t) \right| \le \frac{C_0}{\sqrt{-\log\eps}}
\]
for any $\eps>0$.
\end{lemma}
\begin{proof}
By Lemma~\ref{lem:pdimexists} we have $\delta=\delta^{uu}+\delta^u$, thus by Lemma~\ref{lem:FlGn} we have
\[
-\delta\log\eps \approx \delta^{u} \log F_{m_\eps} + \delta^{uu} \log G_{n_\eps}.
\]
This relation, together with the facts that $\log F_{m_\eps}=S_{m_\eps}\log f'\circ\pi$ and $\log G_{n_\eps}=S_{n_\eps}\log\frac{\partial g}{\partial y}$, and Lemma~\ref{lem:Cepssum} yield
\[
\begin{split}
\log\mu_{\varphi}(C_{\eps}(x,y))-\delta\log\eps
&\approx 
S_{n_\eps}(\varphi-\psi\circ\pi+\delta^{uu}\log \frac{\partial g}{\partial y}) + S_{m_\eps}(\psi\circ\pi+\delta^{u}\log f'\circ\pi)\\
&=
S_{n_\eps}\phi_1 + S_{m_\eps}\phi_2.
\end{split}
\]

Therefore, there exists a constant $C_0$ finite a.e. on $\T^2$ such that for any $\eps$ and $t\in[0,1]$, we have
\[
|N_\eps'(t)-N_\eps''(t)| \le \frac{C_0}{\sqrt{-\log\eps}}.
\]
\end{proof}

To complete the proof of the main theorem we are left to prove the convergence of the process $N_\eps''$ toward a (possibly degenerate) Wiener process. Since $\phi_1$ and $\phi_2$ have a good regularity and are centered it is well known that their Birkhoff sums follow a central limit theorem. However a problem arise here. The ``times'' $n_\eps$ and $m_\eps$ are not constant but they depend on the point.


\section{Invariance principle, random change of time}\label{section-proba-wip}

The invariance principle consists in an approximation of all the trajectory of the processes $(S_n\phi_1)$ and $(S_m\phi_2)$ by a Brownian motion, and this is what we need in a first step. Then, a random change of time in the process will give us back $N_\eps''$.
Observe that it is sufficient to show the convergence in distribution along the subsequence $\eps=e^{-k}$, that is the convergence of the process $\CX_k=N_{e^{-k}}''$ in the Skorohod topology.

\subsection{Invariance principle}

Let $\phi\colon \T^2\to\R^2$ defined by $\phi=(\phi_1,\phi_2)$. The function $\phi$ has stretched exponential decay of the variation $\var_n \phi$. Hence, if we set $S_n\phi=(S_n\phi_1,S_n\phi_2)$, the central limit theorem holds for $S_{n}\phi$.
Denote by $Q$ the limiting covariance matrix of $\frac{1}{\sqrt{n}}S_n\phi$.
Define the process $\CY_k$ by 
\[
\CY_k(t)=\frac{1}{\sqrt{k}}\left(S_{\lfloor kt\rfloor}\phi+(kt-\lfloor kt\rfloor)\phi\circ T^{\lfloor kt\rfloor}\right).
\]
We denote by $\CC$ the space $C([0,1],\R)$ endowed with the topology of uniform convergence.

The weak invariance principle, or functional central limit theorem, is well known in this setting.

\begin{theorem}[WIP, folklore]\label{thm:wip}
The process $\CY_k$ converges in distribution in $\CC^2$ to a two-dimensional brownian motion $\CB=(\CB_t)_{t\in[0,1]}$ with covariance matrix $Q$.
\end{theorem}
Note that $\CB$ (and also $\CY_{k}$) is continuous, hence the Skorohod topology coincides with the topology of uniform convergence.
We remark that the weak invariance principle for vector valued processes is not present in the literature, although it is a part of the folklore. We were indeed not able to give a proper reference, even in this ideal context of uniformly expanding maps with H\"older potential. 
For the sake of completeness one can always invoke the almost sure invariance principle for vector valued observables~\cite{asip}, which implies immediately the weak invariance principle that we need.

Writing $Q=U\Lambda U^*$ for some orthogonal matrix $U$ and $\Lambda=\diag(\sigma_1^2,\sigma_2^2)$, we have that  $\CW:=U^*\CB=(\sigma_1 W_1,\sigma_2 W_2)$, where $W_1$ and $W_2$ are two independent standard Wiener processes.

\subsection{Random change of time and conclusion}

If $n_\eps$ and $m_\eps$ were independent and independent of the process $(\CY_k)$ then we could conclude by direct computation, but these independencies are generally false. The good strategy is to make a random change of time in this process.
We follow the general line of Billingsley (\cite{billingsley}, Theorem 14.4). 
The setting here is a bit different: two dimensional time, no need for Skorohod topology.

\subsubsection{Existence of the limiting distribution.}

Fix $a>1/\lambda^u$.
Let $\CZ_k$ be the process in $C(|0,a]^2,\R^2)$ defined by
$$
\disp \CZ_k(t_1,t_2) = \left( \CY_{k,1}(t_1),\CY_{k,2}(t_2)\right)
$$ 
for any $(t_1,t_2)\in[0,a]^2$. Let $\tilde\nu_k(t)=(n_{e^{-kt}},m_{e^{-kt}})$.
The real functions $\tilde\nu_{k,i}(t)$, $i=1,2$, are not continuous in $t$. We define $\nu_{k,i}(t)$ as the continuous function obtained from $\tilde\nu_{k,i}(t)$ by linear interpolation at the jump points. Namely, $\nu_{k,i}$ is continuous, affine by part, and coincides with $\tilde\nu_{k,i}$ at the jump points.

Let $\theta_1=1/\lambda^{uu}$, $\theta_2=1/\lambda^{u}$ and define the random element $\Phi_k\in C([0,1]^2,[0,a]^2)$ by 
\[
\Phi_k(t_1,t_2) = \begin{cases}
( \nu_{k,1}(t_1) /k,  \nu_{k,2}(t_2)/k) &\text{ if } \nu_{k,1}(1)/k\le a\text{ and } \nu_{k,2}(1)/k\le a\\
( \theta_1 t_1,\theta_2 t_2) &\text{ otherwise}
\end{cases}
\]

Let $\beta\colon C([0,1]^2)\to C([0,1])$ defined by $\beta(u)(t)=u(t,t)$ and $\gamma\colon C([0,1],\R^2)\to C([0,1],\R)$ defined by $\gamma(u)(t)=u_1(t)+u_2(t)$.
Note that 
\[
\CX_k=\beta(\gamma(\CZ_k\circ\Phi_k)) +O(\frac1{\sqrt{k}}),
\]
whenever the condition in the definition of $\Phi_{k}$ holds (both times are less than $a$), which happens eventually almost surely.
\begin{lemma}\label{lem:mnt}
The processes $(\frac{n_{\eps^{t_1}}}{-\log\eps})_{t_1\in[0,1]}$ and  $((\frac{m_{\eps^{t_2}}}{-\log\eps})_{t_2\in[0,1]}$ converge in probability in $\C$, respectively,  to $(\frac{t_1}{\lambda^{uu}})$ and $(\frac{t_2}{\lambda^{u}})$.
\end{lemma}
\begin{proof}
By Lemma~\ref{lem:mn}, almost everywhere, for any  $t_1\in[0,1]$, $\frac{n_{\eps^{t_1}}}{-\log\eps}$ converges to $(\frac{t_1}{\lambda^{uu}})$. Since the process is positive and nondecreasing in $t_1$, it follows from Dini's (or P\'olya's) theorem that the convergence is uniform. Hence the process converges almost surely in $\C$, hence in probability.
The same is true for $m_\eps$.
\end{proof}
By Lemma~\ref{lem:mnt} the map $\Phi_k$ converges almost surely in uniform norm to the map $\Phi$ defined by $\Phi(t_1,t_2)=(\theta_1 t_1,\theta_2 t_2)$ for any $(t_1,t_2)\in[0,1]^2$.

Define the continuous mapping $h$ from $C([0,a],\R^2)$ to $C([0,a]^2,\R^2)$ by
\[
h(y)(t_1,t_2)=(y_1(t_1),y_2(t_2))
,\quad y\in C([0,a],\R^2).
\]

\begin{lemma} 
The process $(\CZ_k)$ converges in distribution to $\CZ=h(\B)$.
\end{lemma}
\begin{proof}
We have $\CZ_k=h(\CY_k)$, and by continuity we get that $\CZ_k$ converges in distribution to $h(\CB)$.
\end{proof}

Since $\CZ_k$ converges to $\CZ$ in distribution and $\Phi_k$ converges to (the deterministic) $\Phi$ in probability, the couple $(\CZ_k,\Phi_k)$ converges to $(\CZ,\Phi)$ (\cite{billingsley}, Theorem 3.9). By continuity of the composition we conclude that $\CZ_k\circ\Phi_k$ converges in distribution to $\CZ\circ\Phi$.
By continuity again we finally get that $\CX_k$ converges in distribution to 
\[
\CX=\beta(\gamma(h(\CB)\circ\Phi)).
\]

\subsubsection{The limit is a Wiener process.}

To finish the proof we are left to characterize the limiting process $\CX$.
Denote the transfer matrix by $U=(u_{ij})$. Note that $\theta_1<\theta_2$.
For any $t\in[0,1]$ we have 
\[
\begin{split}
\CX(t)
&=\beta(\gamma(h(\CB)\circ\Phi))(t)\\
&=
h_1(U\CW)(\theta_1 t,\theta_2 t) + h_2(U\CW)(\theta_1 t,\theta_2 t) \\
&=
u_{11}\sigma_1 W_1(\theta_1 t)+u_{12}\sigma_2 W_2(\theta_1 t) +
u_{21}\sigma_1 W_1(\theta_2 t)+u_{22}\sigma_2 W_2(\theta_2 t) \\
&=
(u_{11}+u_{21})\sigma_1 W_1(\theta_1 t)+(u_{12}+u_{22})\sigma_2 W_2(\theta_1 t) + \\
&\quad\quad + u_{21}\sigma_1 (W_1(\theta_2 t)-W_1(\theta_1 t))+u_{22}\sigma_2 (W_2(\theta_2 t)-W_2(\theta_1 t)) 
\end{split}
\]
By independence of the processes $W_i$ and independence of their increments, we get that $\CX(t)$ is again a Wiener process, its variance is

\begin{eqnarray}
\sigma^2
&:=&
\var \CX(1)\nonumber \\
&=&
((u_{11}+u_{21})\sigma_1)^2\theta_1
+
(u_{12}+u_{22})\sigma_2)^2\theta_1
+
(u_{21}\sigma_1)^2(\theta_2-\theta_1)
+
(u_{22}\sigma_2)^2(\theta_2-\theta_1).
\nonumber \\
&&\label{equ-variance-final-dim2}
\end{eqnarray}

\begin{remark}\label{rem:varzero}
We remark that the variance vanishes if and only if 
\[
\left\{
\begin{matrix}
u_{11}\sigma_1+u_{21}\sigma_1 & =0\\
u_{12}\sigma_2+u_{22}\sigma_2 & =0\\
u_{21}\sigma_1 & =0\\
u_{22}\sigma_2 & =0
\end{matrix}
\right.
\quad\Longleftrightarrow\quad
U \left(\begin{matrix} \sigma_1\\ \sigma_2 \end{matrix} \right)= 0,
\]
which is equivalent to $\sigma_1=\sigma_2=0$ since the matrix $U$ is invertible.
This is equivalent to the fact that the covariance matrix $Q=0$, which happens if and only if both $\phi_1$ and $\phi_2$ are cohomologous to zero. Then, we use Proposition \ref{pro:varzero-acim}.
\end{remark}

We finally have the conclusion: the process $N_\eps$ converges in the Skorohod topology to a Wiener process $N$ with variance $\sigma^2$. 

\section{Generalizations and open questions}\label{sec:general}

For each of these situations the method developed in the paper gives a version of the theorem.
We compute the exact limiting distribution (i.e. the variance of the limit).
We do not rewrite their proofs in full details since it is very close.

\subsection{Conformal hyperbolic dynamics}

We present two situations of conformal hyperbolic dynamics where our method can be applied verbatim.
We refer to~\cite{barreirabook} for their precise definitions, and also for the estimates concerning the geometry of cylinders and further notions such as invariant measures of full dimension and maximal dimension.

\begin{theorem}
Let $J$ be a repeller of a $C^{1+\alpha}$ transformation $T$, for some $\alpha>0$, such that $T$ is conformal and topologically mixing on $J$, and $\mu$ be the equilibrium measure of a H\"older continuous $\varphi\colon J\to\R$. Denote the asymptotic variance of $\varphi+\frac{h_{\mu_\varphi}}{\lambda_{\mu_\varphi}}\log f'$ by $\sigma_u^2$.

Then the statement of the main theorem holds. The variance of the limit is $\sigma^2:=\frac{\sigma_u^2}{\lambda_{\mu_\varphi}}$, which vanishes iff $\mu$ is the measure of maximal (or full) dimension in $J$.
\end{theorem}

The result is obtained by a simplification of our proof: just remove any dependence in $y$. In particular, one can use formula \eqref{equ-variance-final-dim2} with $u_{21}=u_{22}=u_{12}=0$. 

\begin{theorem}
Let $\Lambda$ be a locally maximal hyperbolic set of a $C^{1+\alpha}$ diffeomorphism $T$, for some $\alpha>0$, such that $T$ is conformal and topologically mixing on $\Lambda$, and $\mu$ be the equilibrium measure of a H\"older continuous $\varphi\colon \Lambda\to\R$.

Denote the asymptotic variance of $\varphi+\frac{h_{\mu_\varphi}}{\lambda_s}\log \|df|E^s\|$ by $\sigma_s^2$.
Denote the asymptotic variance of $\varphi+\frac{h_{\mu_\varphi}}{\lambda_u}\log \|df|E^u\|$ by $\sigma_u^2$.

Then the statement of the main theorem holds. The variance of the limit is $\sigma^2:=\frac{\sigma_s^2}{\lambda_s}+\frac{\sigma_u^2}{\lambda_u}$, which vanishes iff $\mu$ has full dimension in $\Lambda$.
\end{theorem}

\begin{remark}
Although there always exists an invariant measure of maximal dimension in $\Lambda$, it is unlikely that $\Lambda$ supports an invariant measure with full dimension. Indeed, we generically have that $\sup_\mu\dim_H\mu<\dim_H(\Lambda)$.

An interesting situation is for the SRB, or physical measure. When $\Lambda$ is the whole manifold then generically the SRB measure does not have full dimension, in particular the variance $\sigma^2\neq0$.
\end{remark}

The proof here is somehow different. The key point is that there are local product structures, both for coordinates (see \emph{e.g.} \cite{Bowen}) and for Gibbs measures  (see \emph{e.g.} \cite{leplaideur1}). Moreover, if we locally set 
$$\mu_{\varphi}\approx\mu_{\varphi}^{s}\otimes\mu_{\varphi}^u,$$
these two measures $\mu_{\varphi}^{u}$ and $\mu_{\varphi}^{s}$ also satisfy some Gibbs property. 

Using these local coordinates, a ball $B((x,y),\eps)$ can be approximate by a cylinder of the form 
$$C_{-m_{\eps}(x)}^{n_{\eps}(y)}.$$
It is important here to note that the quantity $n_{\eps}$ depends only on the future (the unstable direction, coordinate $y$) and conversely, $-m_{\eps}$ depends only on the past (the stable direction, coordinate $x$). 
Then, using the local product structure for the Gibbs measures we get 
\begin{equation}
\label{equ1-casdim2hyperbol}
\mu_{\varphi}(B(x,y),\eps)\approx S_{n_{\eps}(y)}(\phi_u)(y)+S_{m_{\eps}(x)}(\phi_{s})(x),
\end{equation}
with $\phi_{s}$ ad $\phi_{u}$ H\"older continuous, both cohomologous to $\varphi$, and depending only on past (resp. future) coordinates. 
We also observe that the asymptotic distributions of both terms are independent.
Then adapt Section~\ref{section-proba-wip}.

\subsection{Possible extensions to other dynamical systems}

Our main hypotheses was the uniform expansion and skew product structure. It seems however that these hypotheses can be relaxed and we discuss this point below.

\subsubsection{Non-uniformly expanding maps of an interval.}

The first possibility is to relax the uniformity in the expansion.
There is a vast and still growing literature in this subject. However, these results mainly concern absolutely continuous invariant measures. As already said, these measures have no fluctuations and our result is irrelevant in these cases. For other potentials, the literature is not so large. Basically our method could be applied in principle for maps and their Gibbs measures, such that the (functional) CLT hold for a sufficiently regular class of observables.

It is not clear for the moment if the method could be adapted to conformal ``mostly expanding maps'' as studied by Oliveira and Viana in \cite{Oliveira-Viana}. Note that for these maps, the equilibrium state is not a Gibbs measure but only a non-lacunar Gibbs measure. This seems to be an obstruction to adapt our method. 

We emphasize that for some non-uniformly expanding maps the CLT does not hold in the classical form; for example we could be in the non-standard basin of attraction of the normal law; in that case we could prove a version of our main theorem with a suitable modification of the normalization.
A more difficult task is when we have a convergence to a stable law of some index $\alpha<2$. In that case we believe that our method could be carried out, but some difficulties may arise due to the discontinuity of the paths in non Brownian Levy process.

\subsubsection{Non-conformal without skew product structure}
The second and most challenging situation is for non-conformal maps without the skew product structure. Note that we used two strong consequences of this structure: 1) the Lyapunov splitting exists, without going through a natural extension and 2) the projected measure has the Gibbs property. Still, we believe that the result remains true in general.

\medskip\noindent
{\bf  Conjecture.}
{\it Let $M$ be a compact smooth Riemannian manifold and $T:M\circlearrowleft$ be an Axiom-A diffeomorphism. 
Let $\varphi$ be a H\"older continuous function from $M$ to $\R$. 
Let $\mu_{\varphi}$ be the equilibrium state associated to $\varphi$. 
Let $\delta$ be its Hausdorff dimension.

Then there exists a real number $\s\ge0$ such that the process 
$$
\frac{\log\mu_{\varphi}\left(B(x,\eps^{t})\right)-t\delta\log\eps}{\sqrt{-\log\eps}}
$$ 
converges in $\CD([0,1])$ and in distribution to the process $\sigma W(t)$, where $W$ is the standard Wiener process.}

In particular we believe that the SRB measure of a topologically mixing Anosov diffeomorphism of a compact Riemaniann manifold should enjoy this property, and that the variance will vanishes iff the measure is absolutely continuous.

\bibliographystyle{alpha}
\bibliography{mabiblio}

\end{document}

%% file: monstyle.tex
\catcode`\@=11

\def\ps@headings{\let\@mkboth\markboth
    \def\@oddfoot{}\def\@evenfoot{}
    \def\@evenhead{{\rm \thepage}\hfil { \rm \leftmark}}
    \def\@oddhead{{\rm \rightmark}\hfil {\rm \thepage}}
    \def\chaptermark##1{\markboth{\ifnum \c@secnumdepth >\m@ne
          \@chapapp\ \thechapter. \ \fi ##1}{}}%
    \def\sectionmark##1{\markright{\ifnum \c@secnumdepth >\z@
       \thesection. \ \fi ##1}}}

\ps@headings

\catcode`\@=12

%% file: input.tex
\usepackage{amsmath}
\usepackage{amsthm}
\usepackage{amsfonts}
\usepackage{stmaryrd}
\newtheorem{proposition}{Proposition}[section]
\newtheorem{definition}[proposition]{Definition}

\newtheorem{lemma}[proposition]{Lemma}
\newtheorem{theorem}{Theorem}[section]
\newtheorem*{theorem*}{Theorem}
\newtheorem{corollary}[proposition]{Corollary}

\newdimen\AAdi%
\newbox\AAbo%
\def\AAk#1#2{\setbox\AAbo=\hbox{#2}\AAdi=\wd\AAbo\kern#1\AAdi{}}%
\def\AAr#1#2#3{\setbox\AAbo=\hbox{#2}\AAdi=\ht\AAbo\raise#1\AAdi\hbox{#3}}%
\newcommand {\CB}{{\mathcal {B}}}
\newcommand {\CC}{{\mathcal {C}}}
\renewcommand {\CD}{{\mathcal {D}}}

\newcommand {\CM}{{\cal M}}
\newcommand {\CN}{{\cal N}}

\newcommand {\CP}{{\mathcal {P}}}

\newcommand {\CR}{{\cal R}}

\newcommand {\CT}{{\mathcal {T}}}

\newcommand {\CW}{{\cal W}}
\newcommand {\CX}{{\cal X}}
\newcommand {\CY}{{\cal Y}}
\newcommand {\CZ}{{\cal Z}}

\newcommand{\disp}{\displaystyle}
\newcommand{\eps}{\varepsilon}

\def\m1{{-1}}

\newcommand{\ol}{\overline}

\def\s{\sigma}

\newcommand{\wt}{\widetilde}

\newcommand{\inte}[1]{\stackrel{\circ}{#1}}

\newcommand{\ul}[1]{\underline{#1}}
\newcommand{\deri}[2]{\disp\frac{\partial #1}{\partial #2}}

\newcommand{\N}{\Bbb{N}}
\newcommand{\R}{\Bbb{R}}
\newcommand{\C}{\Bbb{C}}
\newcommand{\Z}{\Bbb{Z}}
\newcommand{\T}{\Bbb{T}}

\theoremstyle{definition}
\newtheorem{remark}{Remark}


%% file: TCLlocaldim10.bbl
\begin{thebibliography}{{B}ow75}

\bibitem[Bar08]{barreirabook}
Luis Barreira.
\newblock {\em Dimension and recurrence in hyperbolic dynamics}, volume 272 of
  {\em Progress in Mathematics}.
\newblock Birkh\"auser Verlag, Basel, 2008.

\bibitem[Bil99]{billingsley}
Patrick Billingsley.
\newblock {\em Convergence of probability measures}.
\newblock Wiley Series in Probability and Statistics: Probability and
  Statistics. John Wiley \& Sons Inc., New York, second edition, 1999.
\newblock A Wiley-Interscience Publication.

\bibitem[{B}ow75]{Bowen}
R.~{B}owen.
\newblock {\em {E}quilibrium {S}tates and the {E}rgodic {T}heory of {A}nosov
  {D}iffeomorphisms}, volume 470 of {\em {L}ecture notes in {M}ath.}
\newblock {S}pringer-{V}erlag, 1975.

\bibitem[BPS99]{Barreira-Pesin-Schmeling}
L.~Barreira, Y.~Pesin, and J.~Schmeling.
\newblock Dimension and product structure of hyperbolic measures.
\newblock {\em Ann. of Math. (2)}, 149(3):755--783, 1999.

\bibitem[CU09]{Chazottes-Ugalde-projmarko2}
Jean-Rene Chazottes and Edgardo Ugalde.
\newblock On the preservation of gibbsianness under symbol amalgamation, 2009.

\bibitem[Lep00]{leplaideur1}
R.~Leplaideur.
\newblock Local product structure for equilibrium states.
\newblock {\em Trans. Amer. Math. Soc.}, 352(4):1889--1912, 2000.

\bibitem[LY85a]{Ledrappier-Young1}
F.~{L}edrappier and L.-S. {Y}oung.
\newblock {T}he metric entropy of diffeomorphisms {P}art {I}:
  {C}haracterization of measures satisfying {P}esin's entropy formula.
\newblock {\em {A}nnals of {M}athematics}, 122:509--539, 1985.

\bibitem[LY85b]{Ledrappier-Young2}
F.~{L}edrappier and L.-S. {Y}oung.
\newblock {T}he metric entropy of diffeomorphisms {P}art {II}: {R}elations
  between entropy,exponents and dimension.
\newblock {\em {A}nnals of {M}athematics}, 122:540--574, 1985.

\bibitem[MN09]{asip}
Ian Melbourne and Matthew Nicol.
\newblock A vector-valued almost sure invariance principle for hyperbolic
  dynamical systems.
\newblock {\em Ann. Probab.}, 37(2):478--505, 2009.

\bibitem[OV08]{Oliveira-Viana}
Krerley Oliveira and Marcelo Viana.
\newblock Thermodynamical formalism for robust classes of potentials and
  non-uniformly hyperbolic maps.
\newblock {\em Ergodic Theory Dynam. Systems}, 28(2):501--533, 2008.

\bibitem[Sau06]{Saussol-rapidmixing}
Beno{\^{\i}}t Saussol.
\newblock Recurrence rate in rapidly mixing dynamical systems.
\newblock {\em Discrete Contin. Dyn. Syst.}, 15(1):259--267, 2006.

\end{thebibliography}
